%% file: main.tex
\documentclass[letterpaper, 10 pt, conference]{ieeeconf}
\IEEEoverridecommandlockouts 

\input{packages}

\title{\LARGE \bf
Trajectory Optimization of Stochastic Systems under Chance Constraints via Set Erosion 
}
\author{Zishun Liu$^*$, Liqian Ma$^*$ and Yongxin Chen
\thanks{$^{*}$The contributions of the first two authors are equal.}
\thanks{The authors are with Georgia Institute of Technology, Atlanta, GA 30332. 
        {\tt\small \{zliu910\}\{mlq\}\{yongchen\}@gatech.edu}}%
}

\begin{document}
\maketitle
\thispagestyle{empty}
\pagestyle{empty}

\begin{abstract}
We study the trajectory optimization problem under chance constraints for continuous-time stochastic systems. To address chance constraints imposed on the entire stochastic trajectory, we propose a framework based on the set erosion strategy, which converts the chance constraints into safety constraints on an eroded subset of the safe set along the corresponding deterministic trajectory. The depth of erosion is captured by the probabilistic bound on the distance between the stochastic trajectory and its deterministic counterpart, for which we utilize a novel and sharp probabilistic bound developed recently. By adopting this framework, a deterministic control input sequence can be obtained, whose feasibility and performance are demonstrated through theoretical analysis. Our framework is compatible with various deterministic optimal control techniques, offering great flexibility and computational efficiency in a wide range of scenarios. To the best of our knowledge, our method provides the first scalable trajectory optimization scheme for high-dimensional stochastic systems under trajectory level chance constraints. We validate the proposed method through two numerical experiments. 
\end{abstract}

\section{Introduction}
\label{sec:introduction}
Safety is a basic requirement for a wide range of real-world dynamical systems, including autonomous vehicles, manipulators, drones, and more \cite{singletary2021safety}. Typically, a trajectory is considered safe if it stays in the safe region over a given time horizon. For deterministic systems under bounded disturbances, a variety of deterministic approaches such as dynamical programming \cite{SB-MC-SH-CJT:17}, barrier functions \cite{prajna2007framework,ames2019control}, and model predictive control \cite{mesbah2016stochastic} have been proposed to provide safety guarantees under the worst case. However, when stochastic disturbances are introduced, deterministic approaches often become inapplicable or overly conservative, as the stochastic disturbance can be unbounded or rarely realize the worst case. To better capture the statistical behavior of stochastic trajectories, the safety requirement is instead expressed as a chance constraint, which enforces a lower bound on the probability that the trajectory remains within the safe region over the entire time horizon.

Two important research directions related to chance constraints are safety verification and controller synthesis. 
Safety verification aims to determine whether a given stochastic system satisfies the chance constraint. Effective methods, including reachability analysis \cite{PM-DC-JL:16}, set-erosion strategy \cite{ma2025safety,liu2024safety}, and martingale-based methods \cite{santoyo2021barrier,jagtap2020formal}, have been well developed to enable non-conservative verification of chance constraints. However, many physical systems are not naturally safe in real-world applications. For such systems, it becomes essential to synthesize controllers that ensure safety under chance constraints throughout the trajectory.

The objective of controller synthesis varies across different control tasks. In this paper, we investigate the stochastic trajectory optimization problem under chance constraints, which aims to find control input sequences that both optimize the stochastic trajectory and ensure satisfaction of the chance constraints.
When the system is linear and subject to Gaussian disturbances, this problem can be addressed by controlling the mean and covariance of the trajectory \cite{5970128,yu2025stochastic}, reducing it to a deterministic optimal control problem with linear constraints. However, these methods fail to guarantee satisfaction of the chance constraint over the entire trajectory in continuous-time settings, and they do not scale well to nonlinear systems. For nonlinear stochastic systems, a variety of approaches have been developed to tackle the trajectory optimization problem under chance constraints.
For instance, dynamic programming-based methods \cite{tassa2014control,9993003} leverage dynamic programming to generate control inputs for the stochastic trajectories, but such methods are not applicable to high-dimensional systems. A series of work \cite{nakka2019trajectory,nakka2022trajectory} approximates the continuous-time stochastic system with a deterministic ODE, and develops a deterministic surrogate of the original stochastic problem. However, in this approach, chance constraints are imposed at individual time points, not at the trajectory level. 
Moreover, these schemes become overly conservative when the chance constraints have a high probability level, e.g., $\prob{\mbox{entire trajectory is safe}}\geq 99.99\%$. Other existing methods, such as model predictive control \cite{kohler2024predictive} and Monte-Carlo motion planning \cite{janson2017monte}, are also applied in practice. However, all of these methods either fail to enforce the chance constraint over the entire trajectory or are not scalable to high dimensional problems.

In this paper, we present a novel framework for solving the continuous-time stochastic trajectory optimization problem under chance constraints. Our method is based on a strategy termed set erosion, which converts the chance constraint on the safe set to a safety constraint of a deterministic trajectory on an eroded subset. The degree of erosion is quantified by the probabilistic bound on the gap between the stochastic trajectory and its deterministic counterpart, for which we adopt a tight probabilistic bound established in our recent work \cite{liu2025safetyPT}. By utilizing this probabilistic bound and the set-erosion strategy, our framework reduces the stochastic trajectory optimization problem to a deterministic one, whose feasibility is insensitive to the specified probability level of the chance constraint, and its feasible solutions are proved to be feasible for the original stochastic problem. The performance of the proposed framework is also discussed under proper assumptions.
Our method is applicable to continuous-time nonlinear stochastic systems and can be readily integrated with a variety of deterministic optimal control methods, offering substantial flexibility and scalability in solving the stochastic trajectory optimization problems.

\textit{Notations.} The set of non-negative real numbers is denoted by $\mathbb{R}_{\geq0}$. We use $\|\cdot\|$ to denote $\ell_2$ norm. Given two sets $A,B\subseteq \R^n$, the Minkowski sum of them is defined by $A\oplus B$, and the Minkowski difference is defined by $A\ominus B$. We use $\mbE$ to denote expectation, $\mbP$ to denote probability, $\mathcal{N}(\mu,\Sigma)$ to denote Gaussian distribution, and $\mathcal{B}^n(r,y)$ to denote the ball $\{x\in\R^n: \|x-y\|\leq r\}$. Given a continuously differentiable vector-valued function $f:\R^n\to \R^m$, we denote the Jacobian of $f$ at $x$ by $D_xf(x)$. For a differentiable scalar-valued function $f:\R^n\to \R$, its gradient at $x$ is denoted by $\nabla f(x)$.

\section{Stochastic Safe Trajectory Optimization} \label{sec: problem}

We study the following continuous-time stochastic system
\begin{equation}\label{sys: c-t ss}
     \dX_{t}=f(X_t,u_t,t)\dt+g_t(X_t)\dW_t,
\end{equation}
where $X_t\in \R^n$ is the system state, $u_t\in\mU$ is an open-loop control input chosen from a bounded set $\mU\subset\R^p$, $g_t(X_t)\in\R^{n\times m}$ is the diffusion term, $W_t\in\R^m$ is the $m$-dimensional Wiener process (Brownian motion) modeling the stochastic disturbance, and $f: \R^n\times\R^p\times\R_{\geq0}\to\R^n$ is a smooth transition function. We assume standard Lipschitz and linear growth conditions \cite[Theorem 5.2.1]{BO:13} to ensure \eqref{sys: c-t ss} has a solution. These assumptions are widely accepted in both scientific studies and engineering.

This paper aims at solving the trajectory optimization problem for the stochastic system \eqref{sys: c-t ds} under safety constraints. 
To formulate this problem, we start with the safety of deterministic trajectories.
Consider the deterministic system
\begin{equation}\label{sys: c-t ds}
    \dot{x}_{t}=f(x_t,u_t,t),
\end{equation}
which can be treated as the noise-free version of the stochastic system \eqref{sys: c-t ss}. Given a terminal time $T$, an initial state $x_0$ and a safe set $\mC\subseteq\R^n$, a deterministic trajectory of \eqref{sys: c-t ds} starting from $x_0$ is \textit{safe} on $[0,T]$ if $x_0\in\mC$ and there exists a control input curve $u_t:[0,T]\to\mU$ such that $x_t\in\mC$ holds for any $t\in[0,T]$.

For deterministic trajectory optimization, the deterministic safety condition can be added to the constraints to ensure trajectory safety. 
 However, when considering the stochastic system \eqref{sys: c-t ss}, the concept of deterministic safety is restrictive since the state $X_t$ is unbounded. In this scenario, we shift our focus to chance constraints to better capture the effect of stochastic noise. Given a $\delta\in[0,1]$, a safe set $\mC\subset\mathbb{R}^n$, an initial state $X_0$ and a terminal time $T$, if $X_0\in\mC$ and there exists a control input curve $u_t: [0,T]\to\mU$ such that:
    \begin{equation} \label{eq: sto safety}
        \prob{X_t\in\mC,~ \forall t\leq T}\geq 1-\delta,
    \end{equation}
    then we say the trajectory $\{X_t:t\in[0,T]\}$ controlled by $u_t$ satisfies \textit{ the chance constraint with probability level of $1-\delta$}. $\delta$ is usually chosen as a small value (e.g., $\delta=10^{-4}$) for better safety satisfaction.

\smallskip

Given the cost function $\mathcal{L}_t(X_t,u_t)$, the terminal cost $\Phi_T(X_T)$ and the initial state $X_0\in\mC$, the stochastic trajectory optimization task with safety guarantee can be formalized as the following stochastic optimization problem with chance constraint \eqref{eq: sto safety} 
\begin{subequations}\label{eq: sto opt traj}
    \begin{align}
        \min_{\bm{X,u}} &J_s(\bm{X,u})= \mbE\left\{\int_{0}^{T} \mL_t(X_{t},u_{t})\dt + \Phi_T(X_T)\right\}\\
        \mbox{s.t.}\quad &\dX_{t}=f(X_t,u_t,t)\dt+g_t(X_t)\dW_t, ~ \mbox{given }X_0, \\
        &\prob{X_t\in\mC,~ \forall t\leq T}\geq 1-\delta, \\
        & u_t\in\mU,
    \end{align}
\end{subequations}
where the optimization variables $\{\bm{X,u}\}=\{(X_t,u_t):t\in[0,T]\}$ denote the state trajectory and the control input curve. Unlike optimal control for deterministic systems, it is very challenging to find a feasible stochastic trajectory for \eqref{eq: sto opt traj}. To make it tractable, existing studies have made progress in developing \textit{deterministic} methods for the trajectory optimization problem \eqref{eq: sto opt traj}, which means the control input $u_t$ is generated from a policy that has no randomness, e.g., \cite{nakka2019trajectory,nakka2022trajectory,9993003}. However, these works only guarantee trajectory-level chance constraints for discrete-time systems. As for continuous-time systems, the constraint $\prob{X_t\in\mC}\geq 1-\delta$ is only imposed at a single given time $t$, but cannot be extended to the entire trajectory over $t\in[0,T]$ since there are unaccountably infinite time steps. The goal of this paper is to solve the problem \eqref{eq: sto opt traj} while overcoming the limitations in these existing works, as formalized below.            

\begin{problem}\label{prob1}
    Consider the continuous-time stochastic system \eqref{sys: c-t ss}. Develop a deterministic framework to solve the stochastic trajectory optimization problem \eqref{eq: sto opt traj} with trajectory-level chance constraint (\ref{eq: sto opt traj}c). 
\end{problem}

\section{Trajectory Optimization Framework via Set Erosion} \label{sec: planning}

Given a stochastic trajectory $X_t$ of \eqref{sys: c-t ss}, define its \textit{associated} deterministic trajectory $x_t$ as the trajectory of \eqref{sys: c-t ds} that has the same initial state and control input as $X_t$ at any time. Since deterministic trajectory optimization has been well studied, it is a natural thought to develop the deterministic method on $x_t$ then apply the result to its associated $X_t$. In this section, we first introduce the set-erosion strategy built on the associated trajectories, then propose our trajectory optimization framework and analyze its feasibility and performance.

We follow \cite{liu2025safetyPT} and impose the boundedness assumptions on $g_t(x)$ and the matrix measure of $D_xf(x,u,t)$, which play an important role in characterizing the evolution of system trajectories.  
\begin{assumption}\label{as: boundness}
    For the CT stochastic system~\eqref{sys: c-t ss}, there exist $c\in\R$ and $\sigma>0$ such that,
    \begin{enumerate}
        \item $\mu(D_xf(x,u,t))\leq c$ for any $t\ge 0$, $u\in\mU$, and $x\in\R^n$.
        \item $g_t(x)g_t(x)^{\top}\preceq \sigma^2 I_n$ for any $t\geq0$ and $x\in\R^n$.
    \end{enumerate}
\end{assumption}
In particular, the system \eqref{sys: c-t ss} is said to be \textit{contractive} if Assumption \ref{as: boundness} holds with $c<0$.

\subsection{Set-Erosion Strategy and Probabilistic Tube} \label{subsec: erosion}

Intuitively, a stochastic trajectory fluctuates around its associated deterministic trajectory with high probability. Therefore, if we erode the safe set $\mC$ with a suitable depth $r_t$ to obtain an eroded subset $\tilde{\mC}_t=\mC\ominus\mathcal{B}^n(r_t,0)$, and control the deterministic $x_t$ to stay within $\tilde{\mC}_t$ at any time, then by applying the same $u_t$ to its associated stochastic trajectory, $X_t$ is expected to stay in $\mC$ with high probability. This strategy is termed \textit{set-erosion}. It has been considered in several existing works \cite{liu2024safety,liu2025safetyPT}, and is utilized in our trajectory optimization method.

The key to eroding the safe set $\mC$ is the appropriate erosion depth $r_{\delta,t}$, which is highly related to the probabilistic tube of the stochastic system. Given a finite time horizon $[0,\,T]$, a probability level $\delta\in(0,1)$ and a curve $r_{\delta,t}:[0,T]\to\R_{\geq0}$, the set $\mathcal{T}=\{(t,y)|0\le t\le T, \|y\|\le r_{\delta,t}\}$ is said to be a \textit{probabilistic tube} (PT) of the stochastic system \eqref{sys: c-t ss} if for any associated trajectories $X_t$ and $x_t$:
\begin{equation} \label{eq: def PT}
\begin{split}
    &\prob{(t,X_t-x_t)\in \mathcal{T}, ~\forall t\leq T}\\=~&\prob{\|X_t-x_t\|\leq r_{\delta,t},~\forall t\leq T}\geq 1-\delta,
\end{split}
\end{equation}     
where $r_{\delta,t}$ is said to be the \textit{radius} of PT. It is provable that when the erosion depth is chosen as the radius of PT, the safety of $x_t$ on $\mC\ominus\mathcal{B}^n(r_{\delta,t},0)$ yields the satisfaction of the chance constraint \eqref{eq: sto safety} for the associated $X_t$. 
Therefore, it is crucial to establish a PT with tight radius $r_{\delta,t}$ for stochastic systems. 
This challenge has been resolved in our recent work \cite{liu2025safetyPT}, where we establish probabilistic tubes with tight $r_{\delta,t}$ for the stochastic system \eqref{sys: c-t ss}. The results are summarized in the following theorem, and a more detailed analysis can be found in \cite{liu2025safetyPT}.

\begin{thm} \label{thm: PT}
    Consider the stochastic system \eqref{sys: c-t ss} and its associated deterministic system \eqref{sys: c-t ds} under Assumption \ref{as: boundness}. Let $X_t$ be the trajectory of \eqref{sys: c-t ss} and $x_t$ be its associated deterministic trajectory over a time horizon $[0,T]$. Given $\delta\in(0,1)$ and tunable parameters $\varepsilon\in(0,1)$ and $\Delta t\in(0,T)$, define
    \begin{equation}\label{eq: r CT}
    \begin{split}
        &r_{\delta,t}= \\
        &\begin{cases}
            e^{ct}\sigma\sqrt{\frac{1-e^{-2cT}}{2c}(\varepsilon_1n+\varepsilon_2\log(1/\delta))},~~c\geq0 \\
            \frac{\sigma(\sqrt{1-e^{2ct}}+\sqrt{e^{-2c\Delta t}-1})}{\sqrt{-2c}}\sqrt{\varepsilon_1n+\varepsilon_2\log\frac{2T}{\delta \Delta t}},~c<0
        \end{cases}
    \end{split}
    \end{equation}
   for any $t\in[0,T]$, where $\varepsilon_1=\frac{\log(\frac{1}{1-\varepsilon^2})}{\varepsilon^2}$ and $\varepsilon_2=\frac{2}{\varepsilon^2}$.
   Then 
   \begin{equation}\label{eq:result}
\mathbb{P}\left(\|X_t-x_t\|\leq r_{\delta,t}, ~\forall t\leq T\right)\geq 1-\delta.
   \end{equation} 
\end{thm}

Notice that $r_{\delta,t}$ in Theorem \ref{thm: PT} is imposed on the entire continuous-time trajectory, so it can be leveraged to build the chance constraint on the trajectory level. Moreover, the derived $r_{\delta,t}$ only has an $\mO(\sqrt{\log(1/\delta)})$ dependence on $\delta$, making it sufficiently tight when $\delta$ is extremely small. Figure \ref{fig:linear eg plots} illustrates the PT of stochastic linear systems $\dd X_t=cX_t\dt+\sigma\dW_t$ with different $c$. It is clear that $r_{\delta,t}$ given by Theorem \ref{thm: PT} keeps tight when $\delta=10^{-3}$.                                
\begin{figure}[t]
	\centering
        \includegraphics[width =0.47\linewidth]{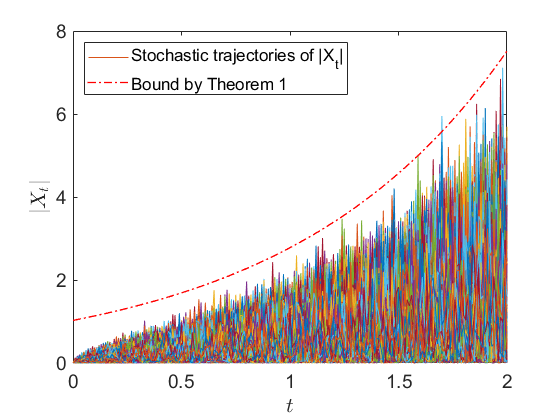}
  \includegraphics[width =0.47\linewidth]{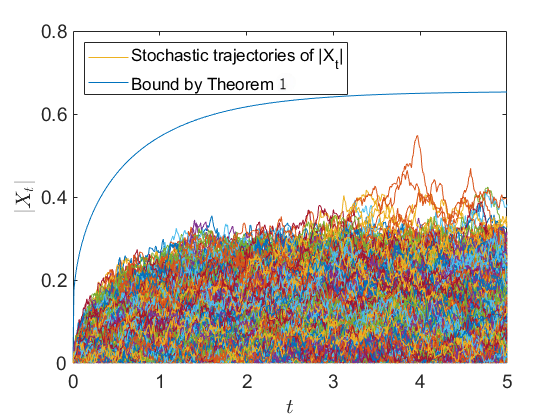}
	\caption{Trajectories of $\|X_t-x_t\|$ of the linear system $\dd X_t=cX_t\dt+\sigma\dW_t$ with $\sigma=\sqrt{0.1}$ and different $c,~T$. Each figure contains 5000 independent trajectories and the $r_{\delta,t}$ calculated by Theorem \ref{thm: PT} with $\delta=10^{-3}$ and $\varepsilon=15/16$. Left: $c=1,~T=2$. Right: $c=-0.5,~T=5$}
\label{fig:linear eg plots}
\end{figure} 

\subsection{Trajectory Optimization Scheme via Set Erosion}
By utilizing the set-erosion strategy with the erosion depth given by Theorem \ref{thm: PT}, the chance constraint (\ref{eq: sto opt traj}c) can be effectively converted to a deterministic constraint. With this deterministic constraint, a deterministic trajectory optimization scheme dual to \eqref{eq: sto opt traj} is established as follows: 
\begin{subequations}\label{eq: opt trj scheme}
    \begin{align}
        &\min_{\bm{x,u}} J_d(\bm{x,u})= \int_{0}^{T} \mL_t(x_{t},u_{t})\dt + \Phi_T(x_T)\\
        \mbox{s.t.}\quad & \dot{x}_{t}=f(x_t,u_t,t),~ x_0=X_0, \\
        &x_t\in \mC\ominus\mathcal{B}^n(r_{\delta,t},0),~\forall t\in[0,T], \label{eq:safety constraint} \\
        &u_t\in\mU.
    \end{align}
\end{subequations}
where the optimization variables $\bm{x}=\{x_t:t\in[0,T]\}$ and $\bm{u}=\{u_t:t\in[0,T]\}$ are deterministic, the stochastic cost $J_s$ in (\ref{eq: sto opt traj}a) is substituted by its deterministic counterpart $J_d$, and $r_{\delta,t}$ is as Theorem \ref{thm: PT}. 
During period $[0,T]$, we solve the deterministic optimization problem \eqref{eq: sto opt traj} and obtain the solution $\{\bm{x}^*,\bm{u}^*\}$, then apply $\bm{u}^*$ to the system \eqref{sys: c-t ss} to acquire a stochastic trajectory $\bm{X}^*=\{X_t^*:t\in[0,T]~|~\bm{u}^*\}$ with the given $X_0\in\mC$. One concern is the feasibility of $\{\bm{X}^*,\bm{u}^*\}$ in solving the stochastic trajectory optimization problem \eqref{eq: sto opt traj}. Following the set-erosion strategy, it turns out that if $\bm{u}^*$ is a feasible solution of \eqref{eq: opt trj scheme}, then $\{\bm{X}^*,\bm{u}^*\}$ is a feasible solution of \eqref{eq: sto opt traj}, as in the following theorem.

\begin{thm} \label{thm: feasibility}
    Consider the stochastic system \eqref{sys: c-t ss} and its associated deterministic system \eqref{sys: c-t ds} satisfying Assumption \ref{as: boundness}. Suppose that $\{\bm{x}^*,\bm{u}^*\}=\{(x_t^*,u_t^*):t\in[0,T]\}$ is a feasible solution of \eqref{eq: opt trj scheme}. Apply $\bm{u}^*$ to (\ref{eq: sto opt traj}b) to obtain a stochastic trajectory $\bm{X}^*=\{X_t^*:t\in[0,T]~|~\bm{u}^*\}$, then $\{\bm{X^*,u}^*\}$ satisfies constraints (\ref{eq: sto opt traj}b)-(\ref{eq: sto opt traj}d).
\end{thm}

\begin{proof}
    Since $\bm{u}^*$ clearly satisfies (\ref{eq: sto opt traj}d), it suffices to show that $\prob{X_t^*\in\mC,~ \forall t\leq T}\geq 1-\delta$. Since $X_0=x_0$ and the same $u_t^*$ is applied to both $X_t^*$ and $x_t^*$, $X_t^*$ and $x_t^*$ are associated trajectories. By the definition of $r_{\delta,t}$ \eqref{eq: def PT}, it holds that
    \begin{equation} \label{eq: P(Xt-xt)}
        \begin{split}
            &\prob{\|X_t^*-x_t^*\|\leq r_{\delta,t},~\forall t\leq T} \\
            =\,&\prob{X_t^*\in\{x_t^*\}\oplus\mathcal{B}^n(r_{\delta,t},0),~\forall t\leq T}
            \geq 1-\delta
        \end{split}
    \end{equation}
    Since $x_t^*\in\mC\ominus\mathcal{B}^n(r_{\delta,t},0)$, \eqref{eq: P(Xt-xt)} yields 
    \begin{equation}
        \begin{split}
            &\prob{X_t^*\in\mC,~ \forall t\leq T} \\
            =\,&\prob{X_t^*\in\mC\ominus\mathcal{B}^n(r_{\delta,t},0)\oplus\mathcal{B}^n(r_{\delta,t},0),~ \forall t\leq T} \\
            \geq\,& \prob{X_t^*\in\{x_t^*\}\oplus\mathcal{B}^n(r_{\delta,t},0),~\forall t\leq T} \geq 1-\delta
        \end{split}
    \end{equation}
    This completes the proof.
\end{proof}

Theorem \ref{thm: feasibility} demonstrates that to find a feasible solution for the problem \eqref{eq: sto opt traj}, it is sufficient to solve the deterministic trajectory optimization problem \eqref{eq: opt trj scheme}. The feasibility of \eqref{eq: opt trj scheme} depends on the value of $r_{\delta,t}$. If $r_{\delta,t}$ is too large, then the constraint (\ref{eq: opt trj scheme}c) can be overly restrictive. Using $r_{\delta,t}$ derived from \ref{thm: PT}, (\ref{eq: opt trj scheme}c) is insensitive to the probability level $\delta$. Regarding the time dependence, $r_{\delta,t}$ exhibits only an $\mO(1)$ dependence on the current time $t$, and an $\mO(\sqrt{\log T})$ dependence on the terminal time $T$, rendering the constraint (\ref{eq: opt trj scheme}c) insensitive to time. When $c>0$, $r_{\delta,t}$ grows exponentially with $t$, making (\ref{eq: opt trj scheme}c) highly sensitive to time.

Compared to \eqref{eq: sto opt traj}, \eqref{eq: opt trj scheme} is a constrained deterministic optimal control problem, which is tractable with various deterministic methods.
In applications, the choice of deterministic methods depends on the realization of (\ref{eq: opt trj scheme}c). For instance, when $\mC$ is defined as the complementary of the unsafe region $\mC_u$, the constraint (\ref{eq: opt trj scheme}c) is equivalent to $x_t\notin \mC_u\oplus\mathcal{B}^{n}(r_{\delta,t},0),~\forall t\in[0,T]$, where $\mC_u\oplus\mathcal{B}^{n}(r_{\delta,t},0)$ is the expansion of unsafe region. When $\mC_u$ is given as the joint of obstacles, $\mC_u\oplus\mathcal{B}^{n}(r_{\delta,t},0)$ can be efficiently approximated to the joint of similar obstacles with a larger size. In this case, \eqref{eq: opt trj scheme} becomes a standard trajectory optimization problem with obstacles, on which various efficient methods have been proposed \cite{tezuka2020time,9987680}.

\subsection{Performance Analysis}
The performance of a trajectory optimization method is typically evaluated by its total cost over the time horizon. For stochastic cost functions, it is usually challenging to analyze their expectations. We next show that, in problem \eqref{eq: sto opt traj}, the total cost $J_s^*=J_s(\bm{X}^*,\bm{u}^*)$ induced by our method can be quantified when the cost functions $\mL_t(x,u)$ and $\Phi_T(x)$ are $L$-Lipschitz continuous.
\begin{thm} \label{thm: perform}
    Consider problem \eqref{eq: sto opt traj} and \eqref{eq: opt trj scheme} under Assumption \ref{as: boundness}. Suppose that \eqref{eq: opt trj scheme} has at least one feasible solution $\{\bm{x}^*,\bm{u}^*\}=\{(x_t^*,u_t^*):t\in[0,T]\}$, $\mL_t(x,u)$ is $L$-Lipschitz, and $\Phi_T(x)$ is $L_T$-Lipschitz for some constants $L,L_T>0$. Let $\bm{X}^*=\{X_t^*:t\in[0,T]\}$ be the stochastic trajectory of \eqref{sys: c-t ss} with initial state $X_0$ and control input $u^*_t$, $J_s^*=J_s(\bm{X}^*,\bm{u}^*)$ and $J_d^*=J_d(\bm{x}^*,\bm{u}^*)$. Then it holds that 
    \begin{equation*}
    \begin{split}
    J_s^*-J_d^* 
        \leq \int_0^T\!\!\sqrt{\frac{L^2n\sigma^2(e^{2ct}-1)}{2c}}\dt+\!\sqrt{\frac{L_T^2n\sigma^2(e^{2cT}-1)}{2c}}.
    \end{split}
    \end{equation*}
\end{thm}

\bigskip
\begin{proof}

Define $V_t = \|X_t^*-x_t^*\|^2$, then a direct application of the Ito's Lemma \cite{sarkka2019applied} yields
    \begin{equation*}
    \begin{split}
         \dd V_t & = 2\left(X_t^*-x_t^*\right)^{\top}(f(X^*_t,u_t,t)-f(x^*_t,u_t,t))\dt \\ & + \mathrm{tr}(g_t(X_t^*)^{\top}g_t(X_t^*))\dt + 2 (X_t^*-x_t^*)^{\top}g_t(X_t^*)\dW_t. \label{eq: dV_t Ito}
    \end{split}
    \end{equation*}
Based on the Fokker–Planck equation \cite{sarkka2019applied}, $\mbE(V_t)$ satisfies
\begin{equation}\label{eq: dE}
\begin{split}
    \frac{\dd \mathbb{E}(V_t)}{\dt}=&2\mbE[\left(X_t^*-x_t^*\right)^{\top}(f(X_t^*,u_t,t)-f(x_t^*,u_t,t))] \\
    &+\mbE[\mathrm{tr}(g_t(X_t^*)^{\top}g_t(X_t^*))].
\end{split}
\end{equation}
According to \cite{FB-CTDS}, Assumption \ref{as: boundness} implies that $(x-y)^{\top}(f(x,u,t)-f(y,u,t))\leq c\|x-y\|^2$ holds for any $x,y\in\R^n$. Therefore, \eqref{eq: dE} can be upper bounded by
\begin{equation}
   \frac{\dd \mathbb{E}(V_t)}{\dt} \leq 2c\,\mathbb{E}(V_t) + n\sigma^2, \quad V_0=0.
\end{equation}
By Gr\"{o}wall Inequality, it follows the expectation bound
\begin{equation} \label{eq: E-bound}
    \mbE(\|X_t-x_t\|^2)=\mbE(V_t)\leq \frac{n\sigma^2(e^{2ct}-1)}{2c}. 
\end{equation}

Now, we are ready to analyze the cost. By Lipschitz continuity, we have 
\begin{equation}
    \mbE(\mL_t(X_t^*,u_t^*))-\mL_t(x_t^*,u_t^*)\leq L\mbE(\|X_t^*-x_t^*\|),
\end{equation}
and the same property holds for $\Phi_T(x)$. By \eqref{eq: E-bound},
\begin{equation}
    \begin{split}
        L\mbE(\|X_t^*-x_t^*\|)
        \leq \sqrt{\frac{L^2n\sigma^2(e^{2ct}-1)}{2c}}
    \end{split}
\end{equation}

Therefore,
\begin{equation} \label{eq: cost gap}
    \begin{split}
    &J_s^*-J_d^* =\int_{0}^T\left(\mbE(\mL_t(X_t^*,u_t^*))-\mL_t(x_t^*,u_t^*)\right)\dt \\
    &+\mbE(\Phi_T(X_T^*,u_T^*))-\Phi_T(x_T^*,u_T^*)\\
        \leq& \int_0^T\sqrt{\frac{L^2n\sigma^2(e^{2ct}-1)}{2c}}\dt+\sqrt{\frac{L_T^2n\sigma^2(e^{2cT}-1)}{2c}}
    \end{split}
\end{equation}
This completes the proof.
\end{proof}
  
When $c>0$, the integral in the last row of \eqref{eq: cost gap} has an analytical result
$\sqrt{\frac{L^2n\sigma^2}{2c^3}}\left(\sqrt{e^{2cT}-1}-\tan^{-1}(\sqrt{e^{2cT-1}})\right),$
which grows exponentially with respect to $T$. When $c=0$, the upper bound of \eqref{eq: cost gap} reduces to $\frac{2\sqrt{L^2n\sigma^2}}{3}T^{3/2}+\sqrt{L_T^2n\sigma^2T}$. When $c<0$, \eqref{eq: cost gap} is upper bounded by $\sqrt{\frac{L^2n\sigma^2}{-2c}}T$ plus a constant. If we take into account the average cost gap $\frac{1}{T}\left(J_s^*-J_d^*\right)$, the average cost gap diverges with $T$ when $c\geq0$ and converges to a finite value when $c<0$. 

Furthermore, when the systems \eqref{sys: c-t ss} and \eqref{sys: c-t ds} are linear, that is, $f(x_t,u_t,t)=A_tx_t+B_tu_t$ with matrices $A\in\R^{n\times n}$ and $B_t\in\R^{n\times p}$, the total cost can be quantified under the more relaxed assumption that the cost functions are $L$-smooth. See the following corollary.

\begin{corollary} \label{coro: L-smooth}
     Consider the problem \eqref{eq: sto opt traj} and \eqref{eq: opt trj scheme}. Suppose that the systems \eqref{sys: c-t ss} and \eqref{sys: c-t ds} are linear and under Assumption \ref{as: boundness}, $\mL_t(x,u)$ is $L$-smooth, and $\Phi_T(x)$ is $L_T$-smooth for some constants $L,L_T>0$. Let $\bm{X}^*=\{X_t^*:t\in[0,T]\}$, $J_s^*$, $\{\bm{x}^*,\bm{u}^*\}$, and $J_d^*$ follow the same definition as those in Theorem \ref{thm: perform}. Then it holds that 
    \begin{equation*}
    \begin{split}
       J_s^*-J_d^*
        \leq \frac{Ln\sigma^2(\frac{e^{2cT}-1}{2c}-T)}{4c}+\frac{L_Tn\sigma^2(e^{2cT}-1)}{4c}.
    \end{split}
    \end{equation*}
\end{corollary}

\bigskip
\begin{proof}
    By the property of $L$-smoothness, we know that
    \begin{equation} \label{eq: L-smooth}
        \begin{split}
            &\mbE(\mL_t(X_t^*,u_t^*))\leq \mL_t(x_t^*,u_t)\\
            &+\innerp{\nabla_{x_t^*}(x_t^*,u_t^*),~ \mbE(X_t^*)-x_t^*}+\frac{L}{2}\mbE(\|X_t^*-x_t^*\|^2).
        \end{split}
    \end{equation}
    The same property holds for $\Phi_T(x)$. For associated trajectories of linear systems, it is well known that $\mbE(X_t^*)=x_t$, so $\innerp{\nabla_{x_t^*}(x_t^*,u_t^*),~ \mbE(X_t^*)-x_t^*}=0$. Apply \eqref{eq: E-bound} to \eqref{eq: L-smooth}, we get
    \begin{equation}
        \mbE(\mL_t(X_t^*,u_t^*))-\mL_t(x_t^*,u_t)\leq \frac{Ln\sigma^2(e^{2ct}-1)}{4c}
    \end{equation}
Therefore,
\begin{equation}\label{eq: cost gap L-smooth}
    \begin{split}
       & J_s^*-J_d^*\\
        \leq &\int_0^T\frac{Ln\sigma^2(e^{2ct}-1)}{4c}\dt+\frac{L_Tn\sigma^2(e^{2cT}-1)}{4c} \\
        =&\frac{Ln\sigma^2(\frac{e^{2cT}-1}{2c}-T)}{4c}+\frac{L_Tn\sigma^2(e^{2cT}-1)}{4c}.
    \end{split}
\end{equation}
This completes the proof.
\end{proof}

Especially, when $\mL_t(x,u)=\|x\|_Q^2+\|u\|_R^2$ and $\Phi_T(x)=\|x\|_S^2$ where $Q,R,S$ are positive definite matrices, the problem \eqref{eq: sto opt traj} becomes a stochastic linear quadratic control problem under chance constraint. 
In this case, the $L$-smoothness condition is satisfied with $L=2\lambda_Q$ and $L_T=2\lambda_S$, where $\lambda_Q$ and $\lambda_S$ are the maximal eigenvalues of $Q$ and $S$, and a straightforward corollary of \eqref{eq: cost gap L-smooth} yields
\begin{equation*}
    J_s^*-J_d^*\leq \frac{\lambda_Qn\sigma^2(\frac{e^{2cT}-1}{2c}-T)}{2c}+\frac{\lambda_Sn\sigma^2(e^{2cT}-1)}{2c}.
\end{equation*}

\section{Case Studies}\label{sec: case}
The proposed stochastic trajectory optimization algorithm is applied to the collision-free motion planning problem for two representative systems subject to stochastic noise: one linear and one nonlinear. The environments are cluttered with obstacles. The unsafe region is defined as the union of all obstacles, $\mC_u = \bigcup_{i=1}^{N} \mC_u^i$, where $\mC_u^i$ denotes the region occupied by the $i$-th obstacle.

\subsection{3D Double Integrator}\label{sec:double integrator}

Consider the following 3D double integrator system
\begin{equation}\label{eq:continuous_double_integrator_3d_compact}
    \dX_t = 
    \left(
    \underbrace{
    \begin{bmatrix}
    \mathbf{0}_{3 \times 3} & \mathbf{I}_3 \\
    \mathbf{0}_{3 \times 3} & \mathbf{0}_{3 \times 3}
    \end{bmatrix}}_{\mathbf{A}}
    X_t +
    \underbrace{
    \begin{bmatrix}
    \mathbf{0}_{3 \times 3} \\
    \frac{1}{m}\mathbf{I}_3
    \end{bmatrix}}_{\mathbf{B}}
    u_t
    \right)\dt + g_t \dW_t,
\end{equation}
where $X_t$ represents the 6-dimensional state vector $[p_x\ p_y\ p_z\ v_x\ v_y\ v_z]^\top$, which are the position and velocity of the mass point, $u_t$ is the control input, $g_t\dW_t$ is the stochastic disturbance with a 6-demensional Wiener process $W_t$. $g_t$ is set to $0.08\mathbf{I}_6$. We set the mass $m=1$. The goal of the stochastic trajectory optimization is to drive the closed-loop system from the start point $[0\ 0\ 0]^\top$ to the end point $[2\ 2\ 2]^\top$ in $T=5$ while avoiding the obstacles with probability $1-10^{-4}$, as shown in Figure~\ref{fig: double integrator}. The start and end velocities are constrained to $\mathbf{0}$. 

\begin{figure}
    \centering
    \includegraphics[width=1\linewidth]{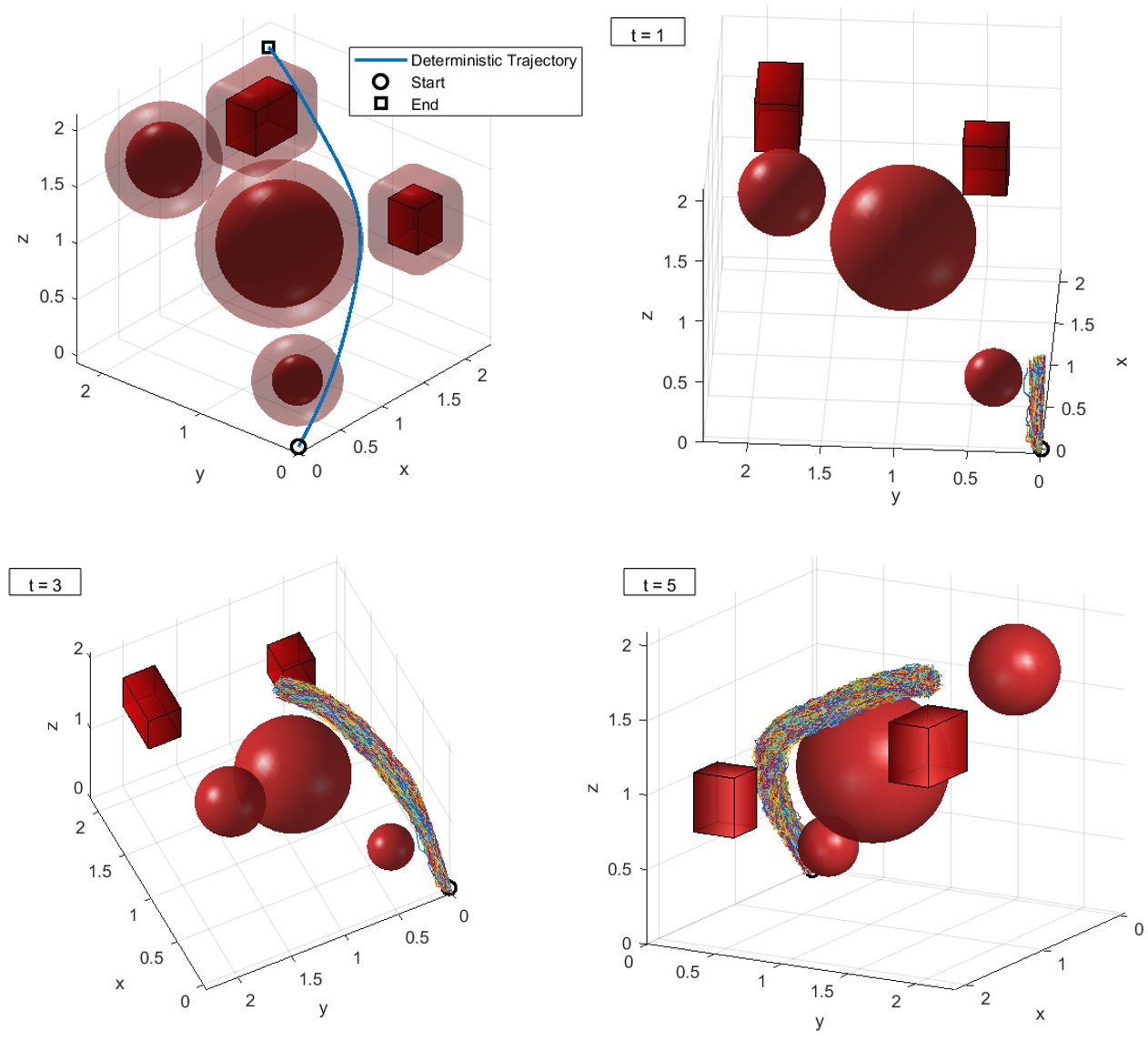} 
    \caption{Trajectory optimization of the double integrator system~\eqref{eq:continuous_double_integrator_3d_compact} with $1-10^{-4}$ guarantee. \textbf{Top left:} The blue curve is the solution of the deterministic trajectory optimization problem. The solid red objects represent the obstacles, while the transparent regions denote the corresponding set erosion. \textbf{Top right}, \textbf{Bottom left} and \textbf{Bottom right}: Visualization of stochastic trajectories at $t = 1$, $3$, and $5$ seconds, from different viewing angles. Each curve represents an independent trajectory of the stochastic system. The stochastic trajectories are simulated with the optimal control input curve.}
    \label{fig: double integrator}
\end{figure}

We construct the deterministic trajectory optimization problem \eqref{eq: opt trj scheme} with $\delta = 10^{-4}$, $\varepsilon=0.9$ and $\Delta t=0.01$. We consider the associated deterministic system and add a linear feedback controller with gain $K=[-10\cdot\mathbf{I}_3\ -5\cdot\mathbf{I}_3]$ to get a contractive closed-loop deterministic system with $A_{\text{cl}} = A+BK$. Following~\cite{chuchu2017simulation}, the optimal contraction rate can be acquired from the following optimization problem
\begin{align*}
        \min_{c_P \in \mathbb{R},\, P \succ 0} \quad & c_P \\
    \text{s.t.} \quad & A_{\text{cl}}^\top P + P A_{\text{cl}} \succeq 2 c_P P.
\end{align*}
This problem can be solved by a bisection search over $c_P$.
We adopt the cost function
\begin{align*}\label{eq:cost func}
    \mL_t \left(x_{t}, (u_{t},x^{\text{ref}}_t)\right) = \|x_t - x^{\text{init}}_t\|^2 + 0.5\|u_t\|^2 + \|x_t - x^{\text{ref}}_t\|^2
\end{align*}
where $x^{\text{init}}_t$ is the initial guess of the state trajectory, and $x^{\text{ref}}_t$ is the reference trajectory included in the control input of the closed-loop system. The safety constraint~\eqref{eq:safety constraint} is implemented by a series of inequality constraints $r_i^2 - \|x_t-a_i\|^2\leq0$, such that $\mathcal{B}(r_i, a_i)$ is a cover of $\mC_u^i\oplus\mathcal{B}(r_{\delta,t},0)$.

To validate the effectiveness of the proposed method, we solve the deterministic trajectory optimization problem with OptimTraj~\cite{Kelly_OptimTraj_Trajectory_Optimization_2022} and simulate $10^5$ stochastic trajectories using the resulting optimal control inputs. The stochastic trajectories are visualized in Figure~\ref{fig: double integrator}. As shown, all simulated trajectories remain collision-free. The cost of the optimal solution, along with the mean cost of the stochastic trajectories, is illustrated in the left panel of Figure~\ref{fig: cost}. The mean cost of the stochastic trajectories remains close to the cost of the optimal solution of the deterministic problem.

\subsection{Unicycle}
Consider the nonlinear kinematic unicycle model
\begin{equation} \label{eq:unicycle dynamics}
    \dX_t = 
    \begin{bmatrix}
        v_t \cos(\theta) \\
        v_t \sin(\theta) \\
        \omega_t + d_t
    \end{bmatrix} \dt + g_t \, \dW_t
\end{equation}
where $X_t = \begin{bmatrix}
    p_{x,t} & p_{y,t} & \theta_t
\end{bmatrix}^\top$ is the state vector, consisting of the unicycle’s position and heading angle. The control inputs are the unicycle's velocity $v_t$ and angular velocity $\omega_t$. $g_t \dW_t$ is the stochastic disturbance with $W_t$ a three-demensional Wiener process. $g_t$ is set to $0.04\mathbf{I}_3$. The goal of the stochastic trajectory optimization is to drive the unicycle from the start point $[0 \quad 0]^\top$ to the end point $[2 \quad 2]^\top$ in $T=3$ while avoiding the obstacles with probability $1-10^{-3}$, as shown in Figure~\ref{fig: unicycle}.

\begin{figure}
    \centering
    \includegraphics[width =0.49\linewidth]{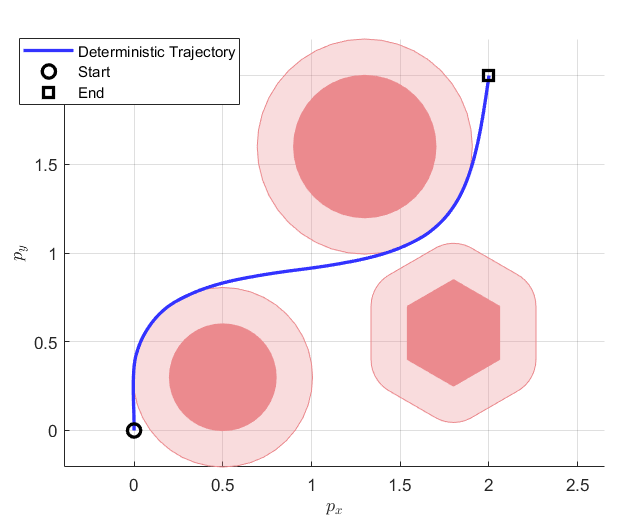}
    \includegraphics[width =0.49\linewidth]{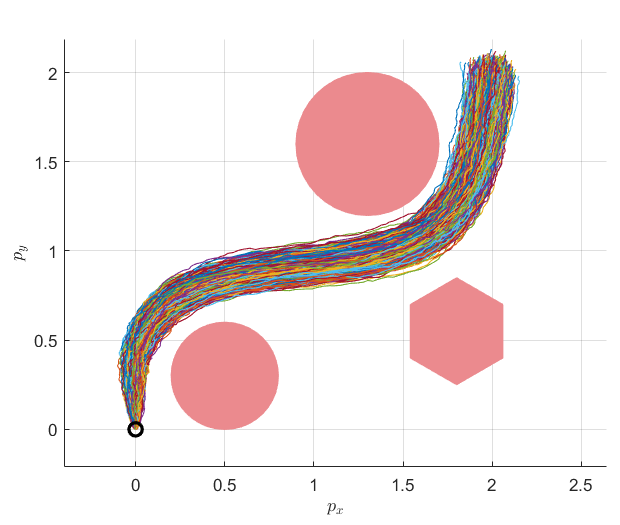}\\
    \caption{Trajectory optimization of the unicycle system~\eqref{eq:unicycle dynamics} with $1-10^{-3}$ guarantee. \textbf{Left:} The solution of the deterministic trajectory optimization problem. The solid red shapes represent the obstacles. The corresponding set erosion is represented as the transparent red areas. \textbf{Right:} Each curve is an independent trajectory of the stochastic system.}
    \label{fig: unicycle}
\end{figure}

We construct the deterministic trajectory optimization problem \eqref{eq: opt trj scheme} with $\delta = 10^{-3}$, $\varepsilon=0.9$ and $\Delta t=0.01$. We consider the associated deterministic system and add a feedback tracking controller 
$
    v_t = v^*_t + K_x\big(\cos\theta_t (p^*_{x,t} - p_{x,t}) + \sin\theta (p^*_{y,t} - p_{y,t})\big),
    \omega_t = \omega^*_t + K_y\big(-\sin\theta_t (p^*_{x,t} - p_{x,t}) + \cos\theta (p^*_{y,t} - p_{y,t}) \big)
                + K_\theta(\theta^*_t-\theta_t)
$ with $K_x=K_y=0.5, K_{\theta}=0.8$
, which forms a contractive closed-loop system. $(v^*_t, \omega^*_t)$ is the feedforward control, $(p^*_{x,t}, p^*_{y,t}, \theta^*_t)$ is the reference trajectory. The cost function and the safety constraint are implemented in the same form as in \ref{sec:double integrator}.

To validate our method, we simulate $10^4$ stochastic trajectories, all of which remain collision-free, as shown in Figure~\ref{fig: unicycle}. The cost function of the optimal solution and the stochastic trajectories are visualized in the right panel of Figure~\ref{fig: cost}. As shown, all simulated trajectories
remain collision-free, and The mean cost of the stochastic trajectories remains close to the cost of the deterministic trajectory.

\begin{figure}
    \centering
    \includegraphics[width =0.49\linewidth]{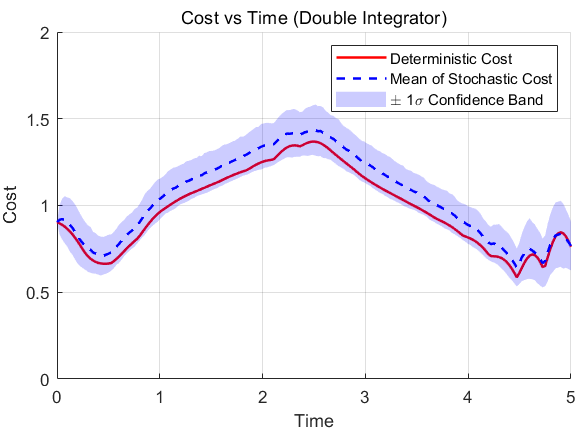}
    \includegraphics[width =0.49\linewidth]{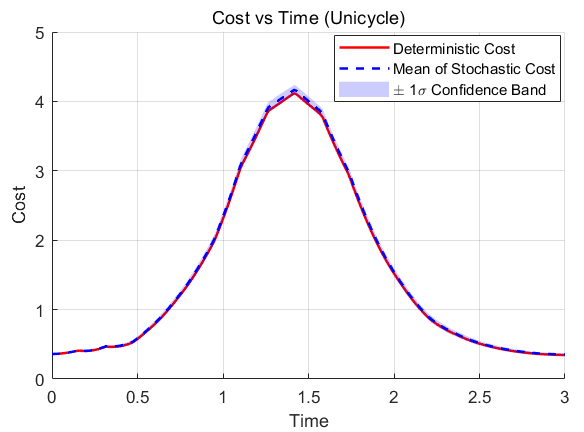}\\
    \caption{Cost functions for both cases. The red curve shows the cost of the optimal deterministic trajectory over time, while the blue dashed line indicates the mean cost of the stochastic trajectories. \textbf{Left:} Double integrator. \textbf{Right:} Unicycle.}
    \label{fig: cost}
\end{figure}

\section{Conclusion} \label{sec: conclusion}
In this paper, we investigated the stochastic trajectory optimization problem under chance constraints for continuous-time systems. Based on the set-erosion strategy, we developed a deterministic trajectory optimization framework in which the chance constraint is converted to a deterministic constraint imposed on an eroded subset of the safe set. The control input obtained by this framework is proved to be a feasible solution of the original problem, and its performance was quantified in several cases. Compared to stochastic schemes, our framework is tractable with a variety of deterministic methods. Two numerical experiments are conducted to validate our method.

\bibliographystyle{ieeetr}
\bibliography{main}    

\end{document}

%% file: packages.tex
\usepackage[utf8]{inputenc} 
\usepackage[T1]{fontenc}    
\usepackage{url}            
\usepackage{amsfonts}       
\usepackage{amsmath} 
\usepackage{amssymb}  
\usepackage{mathrsfs}
\usepackage[ruled]{algorithm2e}
\usepackage{bm}
\usepackage{stfloats}
\usepackage{graphicx}
\usepackage{threeparttable}
\usepackage{nicefrac}       
\usepackage{microtype}      
\usepackage[colorlinks=true, allcolors=black]{hyperref}
\usepackage[normalem]{ulem}

\newcommand{\mC}{\mathcal{C}}

\newcommand{\mU}{\mathcal{U}}

\newcommand{\mL}{\mathcal{L}}

\newcommand{\R}{\mathbb{R}}
\newcommand{\mbE}{\mathbb{E}}
\newcommand{\mbP}{\mathbb{P}}

\newcommand{\mO}{\mathcal{O}}

\newcommand{\dd}{\mathrm{d}}
\newcommand{\dX}{\mathrm{d}X}
\newcommand{\dt}{\mathrm{d}t}
\newcommand{\dW}{\mathrm{d}W}

\newcommand{\innerp}[1]{\langle #1 \rangle}

\newcommand{\prob}[1]{\mathbb{P}\left( #1 \right)}

\renewcommand{\top}{\mathsf{T}}

\usepackage[dvipsnames]{xcolor}

\newtheorem{thm}{Theorem}

\newtheorem{corollary}{Corollary}
\newtheorem{assumption}{Assumption}

\newtheorem{problem}{Problem}
